\documentclass[12pt,reqno]{amsart}%
\usepackage{amsmath, color}
\usepackage[margin=1.5in]{geometry}
\usepackage{graphicx}
\usepackage{amsfonts}
\usepackage{amssymb}%
\usepackage[utf8]{inputenc}
\usepackage{comment}
\usepackage{hyperref}
\usepackage{mathtools}

\newtheorem{theorem}{Theorem}[section]
\theoremstyle{plain}

\newtheorem{corollary}{Corollary}[section]%[theorem]{Corollary}

\newtheorem{proposition}{Proposition}[section]

\numberwithin{equation}{section}
\theoremstyle{definition}
\newtheorem{definition}{Definition}
\theoremstyle{remark}

\allowdisplaybreaks

\newcommand{\VMO}{\mathrm{VMO}}
\title[Hamiltonian stationary equations]{Regularity for Hamiltonian stationary equations in $\mathbb{R}_{n\leq 4}^{n}$
}

\author{Arunima Bhattacharya}
\address{Department of Mathematics, Phillips Hall\\
 the University of North Carolina at Chapel Hill, NC }
\email{arunimab@unc.edu}

\begin{document}

%%%%%%EDITS%%%%%%%%%%%
%5/25/22: added Bao-Chen convex reference to introduction

\begin{abstract}
In this paper, we study the regularity of solutions to the Hamiltonian stationary equation in complex Euclidean space. We show that in dimensions $n\leq 4$, for all values of the Lagrangian phase, any $C^{1,1}$ solution is smooth and derive a $C^{k,\alpha}$ estimate for it, where $k \geq 2$.

\end{abstract}

\maketitle

\section{Introduction}\label{intro}

We study the fourth order Hamiltonian stationary equation, in complex Euclidean space, which locally govern Hamiltonian stationary submanifolds: Let $L$ be a Lagrangian submanifold of a symplectic manifold $(M, \omega)$ with a Riemannian metric $g$ compatible with $\omega$ in the
sense that $\omega(X, Y) = g(JX, Y)$ for an almost complex structure $J$ on $M$. $L$ is  Hamiltonian stationary if its mean curvature
1-form $\omega(H, \cdot)$ is closed and coclosed, that is a harmonic 1-form on
$L$ with respect to the induced metric from $(M,g)$ \cite{Oh}, \cite[pg 1071-1072]{JLS}. In a Calabi-Yau
manifold $(M,\omega,\Omega)$ of complex dimension $n$, this is further equivalent to a scalar equation: the Lagrangian phase function $\Theta$ 
is harmonic.  Here the holomorphic
$n$-form $\Omega$ satisfies $\Omega\wedge\overline{\Omega}={\omega^{n}%
}/{n!}$ and defines $\Theta$ by $\Omega|_{L} = e^{\sqrt{-1}\Theta}d\mu_{L}$.  
The scalar equation follows from the mean curvature formula \cite{HL,Oh,SWJDG} presented below.

In $\mathbb{C}^{n}$, with the standard K\"ahler structure, $\Theta$ can be expressed as the following 
\[
\Theta=\sum_{i=1}^n \arctan \lambda_i
\]
where $u$ is the potential function for a local graphical
representation $L_u = (x,Du)$ and $\lambda_i$ are the eigenvalues of the Hessian of $u$.  The mean curvature vector along
$L_u$ can be written as
\[
\vec{H}=J\nabla_g\Theta
\]
where $\nabla_g$ is the gradient operator of $L_{u}$ (see \cite[(2.19)]{HL}, and $g$ is the induced metric from the Euclidean metric
on $\mathbb{C}^{n}$, which can be written as
\[
g=I_n+(D^{2}u)^{2}.
\]
In this case, the Hamiltonian stationary equation  can be decomposed into two second order equations, given by
\begin{equation}
\Delta_g\Theta=0 \label{hstat0}.
\end{equation}

Our main result, presented below, is a lower-dimensional version ($n \leq 4$) of Theorem 1.2 from Chen–Warren \cite{ChenWarren}, obtained without imposing the additional assumptions specified in (1.6), (1.7), and (1.8) of their paper.

\begin{theorem}\label{main1}
    Suppose that $u\in C^{1,1}(B_1)$ is a weak solution of \eqref{hstat0} on the unit ball $B_1\subset \mathbb{R}^n$ where $n\leq 4$. Then for $k\geq 2$, we have 
    \begin{equation*}
        ||u||_{C^{k, \alpha}(B_{1/2})}\leq C(\alpha, ||D^2u||_{L^{\infty}(B_1)}) .
    \end{equation*}
\end{theorem}

As a corollary, in higher dimensions ($n>4$), we get the following.

\begin{corollary}\label{cor1}
  Suppose that $u\in C^{1,1}(B_1)$ is a weak solution of \eqref{hstat0} on the unit ball $B_1\subset \mathbb{R}^n$ where $n>4$ and $\arctan\lambda_{\min}>(\Theta-\pi)/n$, where $\lambda_{\min}$ is the smallest eigenvalue of the Hessian. Then we have \begin{equation*}
        ||u||_{C^{k, \alpha}(B_{1/2})}\leq C(\alpha, ||D^2u||_{L^{\infty}(B_1)}) .
    \end{equation*}
\end{corollary}

\begin{comment}
Our second result, stated below, extends the above result to all dimensions with the assumption that the Lagrangian phase is critical and supercritical, that is, $|\Theta| \geq (n-2)\frac{\pi}{2}$.

    \begin{theorem}\label{main2}
    Suppose that $u$ is a $ C^{1,1}$ viscosity solution of  \eqref{hstat0} on the unit ball $B_1\subset \mathbb{R}^n$ where $n>4$ and $|\Theta|\geq (n-2)\frac{\pi}{2}$. Then for $k\geq 2$, we have 
    \begin{equation}
        ||u||_{C^{k, \alpha}(B_{1/2})}\leq C(\alpha, ||D^2u||_{L^{\infty}(B_1)}) .
    \end{equation}

\end{theorem}
\end{comment}

\begin{comment}
    \begin{corollary}\label{cor1}
    Suppose that $u\in C^{1,1}(B_1)$ is a weak solution of \eqref{hstat0} on the unit ball $B_1\subset \mathbb{R}^n$ and $D^2u \in \VMO (B_1)$. Then for $k\geq 2$, we have 
    \begin{equation*}
        ||u||_{C^{k, \alpha}(B_{1/2})}\leq C(\alpha, n, ||D^2u||_{L^{\infty}(B_1)}). 
    \end{equation*}
\end{corollary}
\end{comment}

Hamiltonian stationary Lagrangian submanifolds can be viewed as fourth order generalizations of special Lagrangian submanifolds, which are volume-minimizing submanifolds, locally governed by the second order fully nonlinear special Lagrangian equation: 
\begin{equation*}
\sum_{i=1}^n\arctan\lambda_i=c, 
    \end{equation*}
i.e., the Lagrangian phase $\Theta$ is constant. A primary problem is to identify submanifolds that are minimal in a given Hamiltonian isotopy class. Unlike unrestricted minimal surfaces, the constraints of this class permit compact submanifolds to be minimal, for example, the Clifford torus in $\mathbb C^n$. The maximum principle, governing minimal surface equations, no longer applies to these fourth order equations, making them challenging to study.

While the second order special Lagrangian equation has been the subject of extensive study leading to significant advances in understanding existence, regularity, singularities and other geometric properties of its solutions, the regularity theory for the fourth order Hamiltonian stationary equation remains far less complete. The fourth order structure of the equation and the absence of a maximum principle introduce substantial difficulties, leaving many fundamental questions unresolved. For a recent survey of progress on regularity theory for special Lagrangian and Hamiltonian stationary equations, we refer the reader to \cite{chen2022regularity}.

In \cite{ChenWarren}, Chen and Warren proved that weak  $C^{1,1}$ solutions
to the Hamiltonian stationary equation \eqref{hstat0} are smooth as long as one of these conditions are met: The Lagrangian phase $\Theta$ lies in the supercrtical range, namely, $|\Theta|\geq (n-2)\frac{\pi}{2}+\delta$ for $\delta>0$; the potential $u$ is strongly convex; or the Hessian bound of $u$ is sufficiently less than 1. As a consequence, they established that any $C^{1}$-regular Hamiltonian
stationary Lagrangian submanifold in $\mathbb{C}^{n}$ is real analytic. 
In \cite{BW2}, it was shown that any $C^{1,1}$ solution of \eqref{hstat0} is smooth and enjoys interior $C^{2,\alpha}$ estimates in two dimensions, without assuming any additional restrictions. For results on the compactness of Hamiltonian stationary submanifolds in $\mathbb{C}^n$ and more general symplectic manifolds, based on curvature and smoothness estimates derived from the regularity theory of the governing fourth order equations, we refer the reader to \cite{chen2024compactification, chen2024compactness}.

In this paper, we generalize the results of \cite{ChenWarren} and the two-dimensional result in \cite{BW2} to three and four dimensions without requiring any further assumptions. The approach taken in \cite{ChenWarren} and later in \cite{BW2} relies on performing a downward Lewy–Yuan rotation, a rotation technique originally introduced by Yuan in \cite{YY06}. Performing this rotation on the gradient graph allows for the Lagrangian phase to fall in the supercritical range.
In this range, the arctangent operator can be modified into a concave one by exponentiating it. This structural concavity is important, as it allows one to apply the well developed regularity theory for uniformly elliptic, fully nonlinear concave (or convex) equations \cite{CC}. However, this approach does not work in dimensions strictly greater than two unless additional assumptions, such as those used in \cite{ChenWarren}, are imposed. The main difficulty is that the modification of the arctangent operator into a concave one is no longer true when the Lagrangian phase is critical value and subcritical, i.e.,
$\Theta\leq (n-2)\frac{\pi}{2}$. This prevents the application of the regularity theory for concave equations. 

Our proof in dimensions $n\leq 4$, relies on the fact that any degree two homogeneous solution, smooth away from the origin, of the uniformly elliptic special Lagrangian equation is a quadratic polynomial. Our proof goes as follows. Due to the uniform bound on the Hessian of $u$, the fourth order Hamiltonian stationary equation is uniformly elliptic. This uniform ellipticity allows one to apply the De Giorgi–Nash Theorem, which yields Hölder continuity of the Lagrangian phase.
With the phase now known to be Hölder continuous, we study the variable phase second order, uniformly elliptic, fully nonlinear equation, given by \eqref{rs}, with a $C^{\alpha}$ right-hand side. Note that equation \eqref{rs} is known as the Lagrangian mean curvature equation. Through a geometric argument involving rigidity results for the special Lagrangian equation, we show that the Hessian of $u$ lies in the space of functions with vanishing mean oscillation, a key step that enables further regularity. By a result of Lions \cite{Lio83}, weak solutions of this equation are also viscosity solutions. We can therefore apply the regularity theory developed in \cite{BS1} for viscosity solutions with VMO coefficients. This yields $C^{2,\alpha}$ interior estimates for $u$, and smoothness follows from further bootstrapping. \\

\noindent \textbf{Acknowledgements.} The author thanks Yu Yuan and Jacob Ogden for helpful comments. 
The author gratefully acknowledges the support of NSF Grant DMS-2350290, the Simons Foundation grant MPS-TSM-00002933, and a Bill Guthridge fellowship from UNC.

\section{Proof of the main result}
We begin by recalling that a function $u$ is said to solve the Lagrangian mean curvature equation if \begin{equation}
        \sum_{i=1}^n\arctan\lambda_i=\Theta(x) \label{rs}
    \end{equation} 
where $\lambda_i$ are the eignevalues of the Hessian. This equation represents the inhomogeneous form of the special Lagrangian equation. It is referred to as the Lagrangian mean curvature equation because, as shown by Harvey and Lawson \cite{HL}, the phase function $\Theta$ serves as the potential for the mean curvature vector of the gradient graph of $u$. 

Next, we recall the following definition.

\begin{definition}
    [VMO]
Let $\Omega\subset\mathbb{R}^n$. A locally integrable function $v$ is in $\VMO(\Omega)$ or has vanishing mean oscillation with modulus $\omega(r,\Omega)$ if
\[\omega(r,\Omega)=\sup_{x_0\in \Omega,0<r\leq R}\frac{1}{|{B_r(x_0)\cap \Omega}|} \int_{B_r(x_0)\cap \Omega} |v(x)-v_{x_0,r}|\rightarrow 0, \text{ as $r\rightarrow 0$}
\]
where $v_{x_0,r}$ is the average of $v$ over $B_r(x_0)\cap \Omega.$

\end{definition}

We establish VMO estimates for the Hessian of the solution of the Lagrangian mean curvature equation.

\begin{proposition}
\label{prop:VMO}
Suppose that $u\in C^{1,1}(B_1)$ is a viscosity solution of  \eqref{rs} in $B_{1}\subset \mathbb R^n$, where $n\in\{3,4\}$ and $\Theta\in C^{\alpha}(B_{1})$. Then $D^2 u \in \VMO(B_{1/2})$ and the VMO modulus of $D^2u$, denoted by $\omega(r)\rightarrow 0$ as $r\rightarrow 0$.
\end{proposition}

\begin{proof}
We prove by contradiction. Denote \[|D^2u|_{L^{\infty}(B_1)}=\Lambda\] and assume the contrary is true. There exists $\varepsilon>0$ and sequences $\{x_k\rightarrow x_0\}\subset B_{1/2}$, $\{r_k \rightarrow 0\}$, along with a family of $C^{1,1}$ viscosity solutions $\{u_k\}$ of (\ref{rs}) satisfying \[|D^2 u_k|\leq \Lambda,\] such that 
\[
\frac{1}{|B_{r_k}|}\int_{B_{r_k}}|D^2 u_k-(D^2 u_k)_{ x_k,r_k}|\geq \varepsilon.
\] 
Next, blow up $\{u_k\}$. For $| y|\leq \frac{1}{r_k}$, we define
\[ v_k( y)=\frac{ u_k( x_k+r_k y)-\nabla  u_k( x_k)\cdot r_k  y- u_k( x_k)}{r_k^2}.
\]
Observe that $ v_k$ is a viscosity solution of
\[
\sum_{i=1}^n\arctan\lambda_i(D^2 v_k( y))=\Theta( x_k+r_k y).
\]
Since $\Theta$ is H\"{o}lder continuous, the right-hand side converges uniformly to the constant value $\Theta( x_0)$. 
For any fixed $s > 0$, the bound $|D^2 v_k|_{L^{\infty}(B{r_k})} \leq \Lambda$ allows us to extract a $C^{1,\alpha}(B_s)$ convergent subsequence.   
By the diagonalization method, we obtain a subsequence, still denoted by $v_k$, that converges locally uniformly in $C^{1,\alpha}$ on $\mathbb{R}^n$ to  $v$ as $k \to \infty$.
Viscosity solutions are closed under $C^0$ uniform limits and locally uniformly convergent, uniformly elliptic sequences of equations \cite[Proposition 2.9]{CC}, so on any fixed ball, $ v$ is a viscosity solution of the special Lagrangian equation 
\begin{equation}
\sum_{i=1}^n\arctan\lambda_i(D^2 v(y))=\Theta( x_0). \label{slag2}
\end{equation}
Since we also have convergence in $W^{2,p}_{\text{loc}}$, we may apply the $W^{2,\delta}$ estimate from \cite[Proposition 7.4]{CC} to obtain
\[ ||D^2 v_k-D^2 v||_{L^{\delta}(B_{s/2})}\leq C(s)|| v_k- v||_{L^{\infty}(B_s)}.\]
The right-hand side tends to zero as $k \to \infty$. Given that  $|D^2 v_k|$ and $|D^2 v|$ are uniformly bounded, it then follows that for any $p > 0$,
\[||D^2 v_k-D^2 v||_{L^{p}(B_{s/2})} \text{ tends to } 0 \text{ as } k\rightarrow\infty.\]
\smallskip

Next, we show that $v$ is a quadratic polynomial. \\

Case 1: $n=3$. The proof follows from the argument presented on Pg 264 in \cite{YY}. We repeat it here for the sake of completion. Since $v$ is a viscosity solution of the special Lagrangian equation \eqref{slag2}, the gradient graph $(x, v(x))$ is volume minimizing in $\mathbb{R}^3$. Observe that $D^2v_k$
converges to $D^2v$ in $W^{2,3}_{loc}(\mathbb{R}^3)$. Therefore, the gradient graph of $v$ is a cone by the monotonicity formula \cite[Pg 84]{sim} and \cite[Theorem 19.3]{sim}. Since the tangent cone of the gradient graph of $v$ at each point away from the vertex is a 2-dimensional cone cross $\mathbb{R}^1$ \cite[Lemma 35.5]{sim}, the 2-dimensional cone is a linear space by the arguments outlined in the proof of Lemma 2.1 in \cite{YY}. Applying Allard’s regularity result \cite[Theorem 24.2]{sim}
the gradient graph of $v$ is smooth away from the origin. Using \cite[Lemma 2.1]{YY}, we see that $v$ is a quadratic polynomial.
\smallskip

Case 2: $n=4$. This follows similarly if one applies \cite[Theorem 1.1]{NV2} in place of \cite[Lemma 2.1]{YY} to the uniformly elliptic special Lagrangian equation \eqref{slag2}.

\smallskip

Proceeding with the proof, we now combine the  $W^{2,p}_{loc}$ convergence with the fact that $v$ is a quadratic polynomial to obtain

\begin{align*}
0&=\frac{1}{|B_1|}\int_{B_1}|D^2 v-(D^2 v)_{0,1}|\\
&=\lim_{k\rightarrow \infty}\frac{1}{|B_1|}\int_{B_1}|D^2 v_k-(D^2 v_k)_{0,1}|\\
&=\lim_{k\rightarrow \infty}\frac{1}{|B_{r_k}|}\int_{B_{r_k}}|D^2 u_k-(D^2 u_k)_{ x_k,r_k}|\\
&\ge \varepsilon,
\end{align*}
which is a contradiction. 
\end{proof}

\begin{proof} of Theorem \ref{main1}:
The result for $n=2$ follows from \cite{BW2}. We present the proof for the case $n\in\{3,4\}$.
    \ Since $|D^2u|$ is uniformly bounded almost everywhere, equation (\ref{hstat0}) is uniformly elliptic. By the De Giorgi-Nash Theorem, $\Theta \in C^{\alpha}(B_{5/6})$ and noting that $|\Theta|\leq n\pi/2$, we get the following  estimate
    \begin{equation*}
        ||\Theta||_{C^{\alpha}(B_{5/6})}\leq C(||D^2u||_{L^{\infty}(B_{1})}).\label{DGN}
    \end{equation*}
    This means $u\in C^{1,1}$ solves the uniformly elliptic second order fully nonlinear equation \eqref{rs}, with $\Theta \in C^{\alpha}$. By \cite[Corollary 3]{Lio83}, the weak solution $u$ is also a viscosity solution of \eqref{rs}.
    By the above Proposition, $D^2u\in \VMO(B_{4/5})$. Applying \cite[Theorem 2.2]{BS1}, we get $u\in C^{2,\alpha}(B_{3/4})$. Combining this with the above estimate, we get a uniform $C^{2, \alpha}$ estimate for $u$.

    Next, one can apply the bootstrapping argument outlined on Pg 348-349 in \cite{ChenWarren} to obtain smoothness and the desired  $C^{k,\alpha}$ estimate of Theorem \ref{main1}. We briefly describe the argument here for the convenience of the readers.
    Having established that $u \in C^{2, \alpha}$, we consider the divergence-type equation satisfied by $\Theta$, whose coefficients are now in $C^{\alpha}$. By \cite[Theorem 3.13]{HanLin}, this yields $\Theta \in C^{1,\alpha}(B_{2/3})$. We then analyze the equation satisfied by a difference quotient of $\Theta$: For $e_{k\text{ }}$, consider the function
\[
\Theta^{h_{k}}(x)=\frac{\Theta(x+he_{k\text{ }})-\Theta(x)}{h}%
\]
defined on some interior region, for small $h>0.$ Since $\Theta\in
C^{1,\alpha}(B_{2/3})  $, it follows that 
$
| \Theta^{h_{k}}| _{C^{\alpha}(B
_{2/3 -h})  }$ is uniformly bounded. Observe that 
\begin{align*}
\Theta^{h_{k}}(x)  &  =\frac{1}{h}\int_{0}^{1}\frac{d}{dt}F(tD^{2}%
u(x+he_{k\text{ }})+(1-t)D^{2}u(x))dt\\
&  =\frac{1}{h}\int_{0}^{1}g^{ij}\left(  tD^{2}u(x+he_{k\text{ }}%
)+(1-t)D^{2}u(x)\right)  \left(  u_{ij}(x+he_{k\text{ }})-u_{ij}(x)\right)
dt\\
&  =\int_{0}^{1}g^{ij}\left(  tD^{2}u(x+he_{k\text{ }})+(1-t)D^{2}u(x)\right)
\left(  \frac{u_{ij}(x+he_{k\text{ }})-u_{ij}(x)}{h}\right)  dt\\
&  =G^{ij}u_{ij}^{(h_{k})}(x)
\end{align*}
for some uniformly elliptic operator $G^{ij}\partial_{i}\partial_{j}$, with $C^{\alpha}$ coefficients. 
So each $u^{(h_k)}$ satisfies a uniformly elliptic equation of non-divergence type, with right-hand side $\Theta^{h_k} \in C^{\alpha}(B_{2/3 - h})$ and Hölder norms uniformly bounded in $h$. Since each $u^{(h_k)} \in C^{2,\alpha}$, one can apply the non-divergence Schauder estimates \cite[Theorem 6.6]{GT} to obtain a uniform $C^{2,\alpha}$ bound as $h \to 0$.
So for each $1 \leq k \leq n$, we obtain a uniform bound on $||u_k||_{C^{2,\alpha}(B_{1/2})}$, which implies that $u \in C^{3,\alpha}(B_{1/2})$ and $g \in C^{1,\alpha}(B_{1/2})$, with uniform bounds. Using the relation $\Delta_g \theta = 0$, we then deduce
\[
\sqrt{g}g^{ij}\Theta_{ij}=-\partial_{i}( \sqrt{g}g^{ij})
\Theta_{i}\in C^{\alpha}(B_{1/2}).
\] 
Since $\Theta$ satisfies a non-divergence form equation with a Hölder continuous right-hand side, by Schauder theory \cite[Theorem 6.13]{GT}, it follows that $\Theta \in C^{2,\alpha}$. Iterating the previous steps yields higher order regularity estimates in any smaller subregion of the interior of $B_1$.

\end{proof}

\begin{proof} of Corollary \ref{cor1}:
    The above proof holds good in all dimensions, provided Proposition \ref{prop:VMO} holds. In dimensions $n> 4$, the Proposition is valid if $\arctan\lambda_{\min}>(\Theta-\pi)/n$ by \cite[Corollary 1.1]{OY}. 
\end{proof}

\subsection{Comparison with variational formulation }
We outline the analytic framework for the geometric variational problem and provide some remarks comparing equation \eqref{hstat0} with the variational form of the Hamiltonian stationary equation described below.

Let $\Omega \subset \mathbb{R}^n$ be a fixed bounded domain, and let $u : \Omega \rightarrow \mathbb{R}$ be a smooth function. The gradient graph of $u$, $L_{u}=\left\{  \left(  x,Du(x)\right)  :x\in\Omega\right\}  $ is a Lagrangian
$n$-dimensional submanifold in $\mathbb{C}^{n}$, with respect to the complex
structure $J$ defined by the complex coordinates $z_{j}=x_{j}+\sqrt{-1}y_{j}$
for $j=1,\cdots,n$. The volume of $L_{u}$ is given by
\[
F_{\Omega}(u)=\int_{\Omega}\sqrt{\det\left(  I_n+(D^{2}u)^{2}\right)
}dx.
\]
A twice differentiable function $u$ is a critical point of $F_{\Omega}(u)$ under
compactly supported variations of the scalar function $u$ if and only if $u$
satisfies the Euler-Lagrange equation
\begin{equation}
\int_{\Omega}\sqrt{\det g}g^{ij}\delta^{kl}u_{ik}\eta_{jl}\,dx=0\text{
}\ \ \ \text{for all }\eta\in C_{c}^{\infty}(\Omega). \label{vHamstat}%
\end{equation}
The functional $F_\Omega(u)$ is well-defined for functions $u \in W^{2,n}(\Omega)$, making this Sobolev space a natural setting in which to search for critical points. A function $u \in W^{2,n}(\Omega)$ is said to be a weak solution to the variational Hamiltonian stationary equation if it satisfies \eqref{vHamstat}.

However, working directly with the above fourth order equation in double divergence form typically yields much weaker regularity results. A primary reason is that the existing fourth order regularity theory developed in works such as \cite{BW1, BCW, ABfourth, BhaSko} requires the equation to be uniformly elliptic. In other words, one requires the coefficient matrix $\sqrt{\det g}g^{ij}\delta^{kl}$ to satisfy the Legendre ellipticity condition, i.e. for a constant $\Lambda>0$ \begin{equation*}
\label{elliptic:Intro}\frac{\partial}{\partial u_{ik}} \sqrt{\det g}g^{ij}\delta^{kl}(\xi)\sigma
_{ij}\sigma_{kl}\geq\Lambda\left\Vert \sigma\right\Vert ^{2},\text{ $\forall$
}\sigma\text{ $\in S^{n\times n}$ and $\xi\in U$}
\end{equation*} where $S^{n \times n}$ denotes the space of symmetric $n \times n$ matrices. This condition is satisfied by solutions of equation \eqref{vHamstat} whenever the Hessian of $u$ satisfies the smallness condition $||D^2 u||_{L^\infty} \leq 1 - \eta$ for some $\eta > 0$.

On a general Calabi–Yau manifold, although the Lagrangian phase function $\Theta$ remains locally well-defined, it can no longer be written as a sum of arctangent terms, even in Darboux coordinate charts where $L_u$ is expressed as a gradient graph.
 In \cite{BCW}, to establish smoothness of $C^1$-regular Hamiltonian stationary Lagrangian submanifolds in a symplectic manifold, the authors directly dealt with the critical point of
the volume of $L_u$ in an open ball $B\subset\mathbb{R}^{2n}$ equipped
with a Riemannian metric, among nearby competing gradient graphs $L_{t}=\{(x, Du(x)
+t D\eta(x)):x\in\Omega\}$ for compactly supported smooth functions $\eta$. This approach required analyzing a more general class of fourth order equations, structurally similar to \eqref{vHamstat}, but with dependency on lower-order terms arising from the geometric setting. Consequently, outside the Euclidean setting, one is forced to work with the double divergence formulation \eqref{vHamstat}. In contrast, in the Euclidean case, the decomposition into two second order operators, as in \eqref{hstat0} provides stronger results. It is worth noting that the equivalence between these two formulations in $\mathbb{C}^n$ has been established by Chen–Warren under the Hessian smallness condition $||D^2 u||_{L^\infty} \leq 1 - \eta$, see \cite[Theorem 1.1]{ChenWarren}.

\bibliographystyle{amsalpha}
\bibliography{AMS}

\end{document}